\documentclass[11pt, letterpaper]{article}

\usepackage{amssymb,amsmath,amsfonts}
\usepackage{hyperref}
\usepackage{graphicx,color}
\usepackage{calrsfs}		% Caligraphic fonts
\usepackage{amsthm}	% For proof environment
\usepackage[round]{natbib}
\usepackage{bm}

\usepackage{comment}

\usepackage{geometry}
\usepackage{enumitem}

\newcommand{\lc}{[\![}
\newcommand{\rc}{]\!]}
\newcommand{\1}[1]{{\boldsymbol 1_{\{#1\}}}}
\newcommand{\oo}{\boldsymbol 1}

\newcommand{\E}{\mathbb{E}}

\newcommand{\N}{\mathbb{N}}
\renewcommand{\P}{\mathbb{P}}

\newcommand{\R}{\mathbb{R}}

\newcommand{\Ecal}{{\mathcal E}}
\newcommand{\Fcal}{{\mathcal F}}

\newcommand{\Lcal}{{\mathcal L}}

\theoremstyle{plain}
\newtheorem{theorem}{Theorem}

\newtheorem{lemma}[theorem]{Lemma}
\newtheorem{proposition}[theorem]{Proposition}

\theoremstyle{definition}
\newtheorem{definition}[theorem]{Definition}

\theoremstyle{definition}
\newtheorem{remark}[theorem]{Remark}
\newtheorem{important remark}[theorem]{Important remark}
\newtheorem{example}[theorem]{Example}

\numberwithin{equation}{section}
\numberwithin{theorem}{section}

\allowdisplaybreaks

\begin{document}

\title{Stochastic Exponentials and Logarithms on Stochastic Intervals  --- A Survey\thanks{Parts of this note appeared in the unpublished manuscript \cite{Ruf_Larsson}. We thank two anonymous referees and Robert Stelzer for very helpful comments that led to an improvement of the paper. }}
\author{Martin Larsson\thanks{Department of Mathematics, ETH Zurich, R\"amistrasse 101, CH-8092, Zurich, Switzerland. E-mail: martin.larsson@math.ethz.ch} \and
Johannes Ruf\thanks{Department of Mathematics, London School of Economics and Political Science, Columbia House, Houghton St, London WC2A 2AE, United Kingdom. E-mail:     j.ruf@lse.ac.uk}}

%\date{\today}

\maketitle

\begin{abstract}
Stochastic exponentials are defined for semimartingales on stochastic intervals, and stochastic logarithms are defined for semimartingales, up to the first time the semimartingale hits zero continuously.  In the case of (nonnegative) local supermartingales, these two stochastic transformations are inverse to each other. The reciprocal of a stochastic exponential on a stochastic interval is again a stochastic exponential on a stochastic interval.


{\bf MSC2010 subject classification:} Primary 60G99; secondary: 60H10, 60H99.
\end{abstract}

\section{Introduction}
The exponential and logarithmic functions are essential building blocks of classical calculus. As is emphasized by It\^o's formula, in stochastic calculus, second-order terms appear; the appropriate modifications of exponentials and logarithms  lead to  stochastic exponentials and stochastic logarithms. 

This note collects results for the calculus of stochastic exponentials and logarithms of semimartingales, possibly defined on stochastic intervals only. While the results of this note are no doubt well known, we were not able to find a suitable reference. We have found these results rather useful in a number of situations, mostly in the context of measure changes, where it is often convenient to switch between stochastic logarithms and stochastic exponentials. For example, a change of probability measure is given by a nonnegative random variable with expectation equal to one. By taking conditional expectations, this random variable yields a nonnegative martingale $Z$.   On the other hand, Girsanov's theorem, which describes the semimartingale characteristics of some semimartingale under the new measure, is more conveniently stated in terms of the stochastic logarithm of $Z$.

As elaborated, within stochastic calculus, stochastic exponentials and logarithms appear naturally in the context of absolutely continuous changes of measures. If this change of measure is not equivalent, but only absolutely continuous, the corresponding Radon-Nikodym derivative hits zero. Depending on whether it hits zero by a jump or continuously, the corresponding stochastic logarithm may or may not be defined on $[0, \infty)$.  
 This complication motivated us to formulate precise statements concerning the interplay between nonnegative semimartingales and their stochastic logarithms.   %This then allows for an easy formulation of Girsanov's theorem even in the absolutely continuous case.  
 The price to pay is that these stochastic logarithms may only  be defined on stochastic intervals and not on all of $[0, \infty)$.

The reciprocal of a Radon-Nikodym derivative also bears an important interpretation. Provided the original change of measure is equivalent, this reciprocal serves again, under the new measure,  as a Radon-Nikodym derivative; indeed it yields exactly  the original measure.  For this reason, it is convenient to have a description of the dynamics of the reciprocal at hand.  %Imagine, for example, the situation, where the expectation of a random variable is been computed under the new measure by multiplying it with the reciprocal of the Radon-Nikodym derivative. Knowledge of the dynamics of this reciprocal allows to follow these steps.

In general semimartingale theory, which in particular allows for jumps, the notion of stochastic exponential dates back to at least \citet{Doleans_1976}.
 Nowadays, basically any textbook on stochastic calculus introduces this notion. We highlight the survey article \cite{Rheinl_2010}, which reviews well known properties of stochastic exponentials of semimartingales.  In particular, this survey also collects classical conditions for the martingale property of the stochastic exponential. The article of \cite{Kallsen_Shir} provides further interesting identities, especially relating to exponential and logarithmic transforms, a subject which we do not discuss in this note. In contrast to these articles, we especially discuss the definition of stochastic logarithms of general semimartingales, without the assumption of strict positivity. 

We proceed as follows.  In Section~\ref{S:2}, we establish notation and introduce the concept of processes on stochastic intervals. In Section~\ref{S:3}, we define stochastic exponentials and logarithms, discuss their basic properties, and prove that they are inverse to each other.  In Section~\ref{S:4}, we describe the stochastic logarithm of the reciprocal of a stochastic exponential. Finally, in Section~\ref{S:5}, we provide some examples. These examples illustrate that stochastic exponentials of semimartingales, defined on stochastic intervals only, arise naturally.

\section{Notation and processes on stochastic intervals}  \label{S:2}

The following definitions are consistent with those in~\citet{JacodS}, to which the reader is referred for further details. We work on a stochastic basis $(\Omega,\mathcal F, \mathbb F, \P)$, where the filtration $\mathbb F=(\Fcal_t)_{t\ge0}$ is right-continuous but not necessarily augmented with the $\P$-nullsets. Relations between random quantities are understood in the almost sure sense.

Given a process $X=(X_t)_{t\ge0}$, write $X_-$ for the left limit process (limit inferior if a limit does not exist) and $\Delta X=X-X_-$ for its jump process, using the convention $X_{0-}=X_0$.  The corresponding jump measure is denoted by $\mu^X$, and for any (random) function $F: \Omega \times \R_+ \times \R\to\R$, the stochastic integral of $F$ with respect to $\mu^X$ is the process $F*\mu^X$ given by
\[
F*\mu^X_t = \begin{cases} \sum_{s \leq t} F(s, \Delta X_s) \1{\Delta X_s \neq 0}, & \text{if } \sum_{s \leq t} |F(s, \Delta X_s)| \1{\Delta X_s \neq 0} < \infty, \\
+\infty, &\text{otherwise,}\end{cases}\qquad t\ge0.
\]

Semimartingales are required by definition to be right-continuous, almost surely admitting left limits. If $X$ is a semimartingale,  $H\cdot X$ is the stochastic integral of an $X$--integrable process $H$ with respect to $X$.
 
For a stopping time  $\tau$, we let $X^\tau$ denote the process $X$ stopped at~$\tau$, and we define the stochastic interval
\[
\lc0,\tau\lc = \{ (\omega,t)\in\Omega\times\R_+ : 0 \leq  t<\tau(\omega)\}.
\]
Note that stochastic intervals are disjoint from $\Omega\times\{\infty\}$ by definition.

A process $X$ on a stochastic interval $\lc0,\tau\lc$, where $\tau$ is a stopping time, is the restriction to $\lc0,\tau\lc$ of some process. 
In this paper, $\tau$ will be a foretellable time; that is, a $[0,\infty]$--valued stopping time that admits a nondecreasing sequence $(\tau_n)_{n \in \N}$ of stopping times, with $\tau_n<\tau$ almost surely for all $n \in \N$ on the event $\{\tau>0\}$, and $\lim_{n \uparrow \infty} \tau_n = \tau$ almost surely. Such a sequence is called an announcing sequence. Every predictable time is foretellable, and if the stochastic basis is complete the converse also holds; see \citet[Theorem~I.2.15 and I.2.16]{JacodS}.\footnote{In general, the converse implication does not hold. For example, consider the canonical space of c\`adl\`ag paths, equipped with the Skorohod topology and the Wiener measure. Then the first time that the coordinate process crosses a given level is foretellable but not predictable, given the canonical filtration not augmented by the Wiener nullsets.}

If $\tau$ is a foretellable time and $X$ is a process on $\lc0,\tau\lc$, we say that $X$ is a semimartingale (local martingale / local supermartingale) on $\lc0,\tau\lc$ if there exists an announcing sequence $(\tau_n)_{n\in \N}$ for $\tau$ such that $X^{\tau_n}$ is a semimartingale (martingale / supermartingale) for each $n \in \N$.
Basic notions for semimartingales carry over by localization to semimartingales on stochastic intervals. For instance, if $X$ is a semimartingale on $\lc0,\tau\lc$, its quadratic variation process $[X,X]$ and the continuous version $[X,X]^c$ are defined as the processes on $\lc0,\tau\lc$ that satisfy $[X,X]^{\tau_n} = [X^{\tau_n}, X^{\tau_n}]$ and $([X,X]^c)^{\tau_n} = [X^{\tau_n}, X^{\tau_n}]^c$, respectively,  for each $n \in \N$. The jump measure~$\mu^X$ of $X$ is defined analogously, as are stochastic integrals with respect to $X$ (or $\mu^X$). In particular, $H$ is called $X$--integrable if it is $X^{\tau_n}$--integrable for each $n\in\N$, and $H\cdot X$ is defined as the semimartingale on $\lc 0, \tau\lc$ that satisfies $(H \cdot X)^{\tau_n} = H \cdot X^{\tau_n}$ for each $n \in \N$.  We refer to \citet{Maisonneuve1977}, \citet{Jacod_book}, and Appendix~A in~\citet{CFR2011} for further details on local martingales on stochastic intervals.

The following version of the supermartingale convergence theorem is a useful technical tool for studying stochastic exponentials of local supermartingales.  It has been proven, for example, in \citet{CFR2011} via Doob's inequalities. For sake of completeness, we provide an alternative proof here.

\begin{proposition} \label{P:supermg conv}
Let $\tau$ be a foretellable time, and let $X$ be a local supermartingale on $\lc0,\tau\lc$ bounded from below. Then $\lim_{t\uparrow\tau}X_t$ exists in $\R$ and $[X,X]_\tau=\lim_{t\uparrow\tau}[X,X]_t$ is finite.
\end{proposition}

\begin{proof}
	Without loss of generality, we shall assume $X \geq 0$. We define $X' = X \oo_{\lc 0,  \tau\lc}$. We now argue that $X'$ is a supermartingale, which in particular implies that $X'$ allows for a modification with  left limits almost surely. This observation then implies the assertion since the classical supermartingale convergence theorem (see Problem~1.3.16 in \cite{KS1}) yields that $X$ can be closed. This directly implies the existence of a limit at infinity and the convergence of $[X,X]$; in particular then $[X,X]_\tau < \infty$.
	
To prove that $X'$ is a supermartingale, let $(\tau_n)_{n\in \N}$ be an announcing sequence for $\tau$ such that $X^{\tau_n}$ is a supermartingale.  Fix $s,t \geq 0$ with $s<t$. Then, on the event $\{s \geq \tau\}$, we have $\E[X'_t|\Fcal_s] = 0 = X'_s$. On the event $\{s < \tau\}$, Fatou's lemma implies
\begin{align*}
	\E[X_t' | \Fcal_s] = \E\left[\left.\lim_{n \rightarrow \infty} X_t^{\tau_n} \oo_{\{\tau > t\}} \right| \Fcal_s\right] \leq \liminf_{n \rightarrow \infty}  \E[X_t^{\tau_n} | \Fcal_s ] \leq \lim_{n \rightarrow \infty}  X_s^{\tau_n} = X_s = X_s',
\end{align*}
yielding the claim.
\end{proof}

The following corollary will be used below.
\begin{lemma} \label{L:181009}
Let $\tau$ be a foretellable time, and let  $X$ be a local supermartingale on $\lc 0, \tau\lc$  with $\Delta X \geq -1$. Then we have, almost surely, the set identity 
\begin{align} \label{eq:181008}
	\left\{\lim_{t\uparrow\tau} X_t  \text{ does not exist in $\R$} \right\} = \left\{\lim_{t\uparrow\tau}X_t  = -\infty\right\} \cup \left\{[X,X]_{\tau}  = \infty\right\}.
\end{align}
\end{lemma}
\begin{proof}
This statement is proven in Corollary~4.4 of \cite{Larsson:Ruf:convergence}.  For sake of completeness, we provide a proof here of the inclusion ``$\supset$''. For an arbitrary $m\in\N$, define the stopping time $\rho=\inf\{t\ge0\colon X_t\le -m\}$. Then $X^\rho$ is a local supermartingale on $\lc0,\tau\lc$ bounded from below by $-m-1$, whence $[X^\rho, X^\rho]_{\tau} < \infty$ by Proposition~\ref{P:supermg conv}. Since $X$ coincides with $X^\rho$ on $\{X\ge -m\}$, we deduce that
\[
\left\{\lim_{t\uparrow\tau}X_t \text{ exists in $\R$}\right\} \subset \bigcup_{m\in\N} \{X \ge -m\} \subset \left\{[X,X]_{\tau}  < \infty\right\};
\]
hence the inclusion follows.  For the inclusion ``$\subset$'', see Remark~\ref{R:181009} below.
\end{proof}

\section{Stochastic exponentials and logarithms}   \label{S:3}

In this section, we define stochastic exponentials and logarithms, develop some of their properties, and show that they are inverse to each other.

\subsection{Stochastic exponentials}

\begin{definition}[Stochastic exponential]\label{D:stochExp}
Let $\tau$ be a foretellable time, and let $X$ be a semimartingale on $\lc 0,\tau\lc$. % with $\Delta X\ge-1$. 
The \emph{stochastic exponential of $X$} is the process $\mathcal E(X)$ defined by
\begin{equation} \label{eq:D:stochExp}
\mathcal E(X )_t = \exp\left(X_t - \frac{1}{2}[X,X]^c_t \right) \prod_{0<s\le t} (1+\Delta X_s)e^{-\Delta X_s}
\end{equation}
for all $t< \tau$, and by $\mathcal E(X)_t=0$ for all $t\ge\tau$. \qed
\end{definition}

\begin{remark}
Note that for each $n \in \N$, on the interval $\lc 0, \tau_n\lc$ there are only finitely many times $t$ such that $\Delta X_t < -1$.  Moreover, whenever $\Delta X_t\ge-1$ for some $t \in [0, \tau)$, then the corresponding factor in the infinite product in \eqref{eq:D:stochExp} lies in $[0,1]$.  Hence, the  infinite product converges  on the interval $\lc 0, \tau_n\lc$  for each $n \in \N$.
%Strictly speaking, the assumption $\Delta X\ge-1$ (almost surely!) does not exclude larger negative jumps occurring on a nullset. A more pedantic definition would thus replace $\Delta X$ by $\Delta X \vee (-1)$ in~\eqref{eq:D:stochExp}. Furthermore, the condition
 We also emphasize that the stochastic exponential  $\mathcal E(X )$ of a semimartingale $X$ on $\lc0,\tau\lc$ need not be a semimartingale on  $\lc0,\infty\lc$, but only on $\lc 0,\tau\lc$.
\qed
\end{remark}

\begin{remark}
Whenever we have $\Delta X_t<-1$ for some $t \in [0, \tau)$  the resulting stochastic exponential  changes sign at $t$. Such stochastic exponentials appear in the context of signed measures such as in the study of mean-variance hedging strategies; see, for example, \cite{Cerny:Kallsen}.
\qed
\end{remark}

The process $\mathcal E(X)$ is sometimes also called generalized stochastic exponential; see, for example, \cite{Mijatovic:Novak:Urusov}. If $(\tau_n)_{n\in\N}$ is an announcing sequence for $\tau$, then $\Ecal(X)$ of Definition~\ref{D:stochExp} coincides on $\lc0,\tau_n\rc$ with the usual (Dol\'eans-Dade) stochastic exponential of $X^{\tau_n}$. This shows that $\Ecal(X)$ coincides with the classical notion when $\tau=\infty$. Many properties of stochastic exponentials thus remain valid. For instance, if $\Delta X>-1$ then $\Ecal(X)$ is strictly positive on~$\lc0,\tau\lc$. If $\Delta X_t=-1$ for some $t\in[0,\tau)$ then $\Ecal(X)$ jumps to zero at time $t$ and stays there. Also, on $\lc0,\tau\lc$, $\mathcal E(X)$ is the unique solution to the equation
\begin{equation} \label{eq:170203.1}
Z = e^{X_0} + Z_- \cdot X \qquad \text{on $\lc 0,\tau\lc$};
\end{equation}
 see  \citet{Doleans_1976}. We also record the alternative expression
\begin{equation} \label{eq:170206.1}
\mathcal E( X ) = \oo_{\lc0,\tau\lc} \exp\left(X - \frac{1}{2}[X,X]^c - (x - \log|1+x|) * \mu^X\right) (-1)^{\sum_{t \leq \cdot} \1{ \Delta X_t < -1} },
\end{equation}
where we use the convention $-\log(0)=\infty$ and $e^{-\infty}=0$.

The following results relate the convergence of $\Ecal(X)$ to zero  to the behavior of $X$.

\begin{proposition} \label{P:1'''}
	 The following set inclusion holds almost surely:
	 \begin{align*}
	 	\left\{\lim_{t\uparrow\tau}\Ecal(X)_t=0\right\} \subset \left\{\lim_{t\uparrow\tau}X_t  = -\infty\right\} \cup \left\{[X,X]_{\tau}  = \infty\right\} \cup \left\{\Delta X_t = -1\text{ for some $t \in [0, \tau)$} \right\}.
	 \end{align*}
	 Moreover, if we additionally have $\Delta X \geq -1$ and $\limsup_{t \uparrow \tau} X_t < \infty$, then the reverse set inclusion also holds.
%	 In particular, if $X$ is a local supermartingale on $\lc 0, \tau\lc$ with $\Delta X \geq -1$, we have
%	 \begin{align*}
%	 	\left\{\lim_{t\uparrow\tau}\Ecal(X)_t=0\right\} = \left\{\lim_{t\uparrow\tau}X_t  = -\infty\right\} \cup \left\{\Delta X_t = -1\text{ for some $t \in [0, \tau)$} \right\}.
%	 \end{align*}	 
\end{proposition}
\begin{proof} 
Assume we are on the event 
\[
	\left\{\lim_{t\uparrow\tau}\Ecal(X)_t=0\right\} \cap \{[X,X]_{\tau}  < \infty\} \cap \big\{\Delta X_t \neq-1\text{ for all $t \in [0, \tau)$} \big\}.
\]
We need to argue that $\lim_{t\uparrow\tau}X_t  = -\infty$ on this event. To this end, observe that the inequality
\[
	x - \log|1+x| \leq x^2 \quad \text{for all $x \geq -\frac{1}2$}
\]
together with \eqref{eq:170206.1} yield, on this event, that
\begin{align*}
	-\infty &= \lim_{t\uparrow\tau} \left(X_t - \frac{1}{2}[X,X]^c_t - (x - \log|1+x|) * \mu^X_t\right)\\
		&\geq \lim_{t\uparrow\tau} \left(X_t - [X,X]_t - (x - \log|1+x|)\oo_{x < -1/2}  * \mu^X_t\right).
\end{align*}
By assumption, $[X,X]_\tau<\infty$. In particular, $X$ can only have finitely many jumps bounded away from zero. We deduce that the second and third terms on the right-hand side converge, and therefore $\lim_{t\uparrow\tau} X_t = - \infty$. This yields the first set inclusion.

We now assume that $\Delta X\ge-1$ and $\limsup_{t \uparrow \tau} X_t < \infty$, and prove the reverse set inclusion. On the event $\{\text{$\Delta X_t = -1$ for some $t \in [0, \tau)$}\}$, $X$ jumps to zero before $\tau$ and stays there, so that clearly $\lim_{t\uparrow\tau}\Ecal(X)_t=0$. If $\Delta X_t>-1$ for all $t\in[0,\tau)$, then \eqref{eq:170206.1} and the inequality $x-\log(1+x)\ge(x^2 \wedge 1)/4$ for all $x > -1$ give
\[
0 \le \Ecal(X)_t \le\exp\left( X_t - \frac12[X,X]^c_t - \frac14(x^2\wedge1)*\mu^X_t\right), \quad t\in[0,\tau).
\]
On the event $ \{\lim_{t\uparrow\tau}X_t  = -\infty\}$, the right-hand side converges to zero. The same thing happens on the event $\{[X,X]_{\tau}  = \infty\}$, thanks to the assumption that $\limsup_{t \uparrow \tau} X_t < \infty$ and the observation that $[X,X]^c_\tau+(x^2\wedge1)*\mu^X_\tau=\infty$ if and only if $[X,X]_\tau=\infty$. This concludes the proof of the reverse set inclusion.
%, note that if $\Delta X_t = -1$ for some $t \in [0, \tau)$ then by definition of the stochastic exponential we have $\Ecal(X)_t = 0$. Let us next assume that we are on the event  $ \{\lim_{t\uparrow\tau}X_t  = -\infty\}$. Since $\Delta X \geq -1$ and \eqref{eq:170206.1} yield $0\le\Ecal(X) \le e^X$ we get $\lim_{t\uparrow\tau}\Ecal(X)_t=0$. Finally, let us assume to be on the event $\{[X,X]_{\tau}  = \infty\}$. 
%From the assumption that $\limsup_{t \uparrow \tau} X_t < \infty$ and from the definition of the stochastic exponential in \eqref{eq:D:stochExp} it suffices to argue that the product then converges to zero provided that $x^2 * \mu_{\tau} =  \infty$. This is equivalent to $(x^2 \wedge 1) * \mu_\tau = \infty$.  Now, 
%note that
%\[
%	\log(1+x) - x \leq -\frac{1}{4} (x^2 \wedge 1) \qquad  \text{for all $x > -1$};
%\]
%hence, indeed $(\log(1+x) - x) * \mu_{\tau} =  -\infty$, which concludes the proof of the reverse set inclusion.
\end{proof}

\subsection{Stochastic logarithms}
To be able to discuss stochastic logarithms,  recall that for a stopping time $\rho$ and an event $A\in\mathcal F$, the \emph{restriction of $\rho$ to $A$} is given by
\[
\rho(A) = \rho \oo_A + \infty \oo_{A^c}.
\]
Here $\rho(A)$ is a stopping time if and only if $A\in\mathcal F_\rho$.
Define now for a progressively measurable process $Z$ the running infimum of its absolute value by $\underline{Z} = \inf_{t \leq \cdot} |Z_t|$ and the stopping times\footnote{As the filtration might not be augmented by the nullsets, the following definitions might not be stopping times. However, there exist appropriate modifications of these random times which turns them into stopping times; see Appendix~A, in particular, Lemma~A.3, in \cite{Perkowski_Ruf_2014}. We shall always work with these modifications.} 
\begin{align}
\nonumber
\tau_0 &= \inf\left\{ t\ge 0: \underline{Z}_t = 0\right \} ;\\
\label{eq:tauC}
\tau_C &= \tau_0(A_C), \quad\quad A_C = \left\{\underline{Z}_{\tau-}=0\right\};\\
\tau_J &= \tau_0(A_J), \quad\quad A_J = \left\{\underline{Z}_{\tau_0-} > 0\right\}. \nonumber
\end{align}
These stopping times correspond to the two ways in which $Z$ can reach zero: either continuously or by a jump. We have the following well known property of $\tau_C$; see, e.g., Exercise~6.11.b in \citet{Jacod_book}.

\begin{lemma}\label{L:tauCpred}
Fix a progressively measurable process $Z$.
The stopping time $\tau_C$ of \eqref{eq:tauC} is foretellable.
\end{lemma}

\begin{proof}
We claim that an announcing sequence $(\sigma_n)_{n\in\mathbb N}$ for $\tau_C$ is given by
\begin{align*}
	\sigma_n = n\wedge\sigma'_n(A_n), \qquad \sigma'_n = n \wedge \inf\left\{t\ge 0 : \underline{Z}_t  \leq \frac{1}{n} \right\}, \qquad A_n = \left\{\underline{Z}_{\sigma_n'} > 0\right\}.
\end{align*}
To prove this, we first observe that $\sigma_n = n < \infty = \tau_C$ on $A_n^c$ for all $n \in \N$. Moreover, we have $\sigma_n = \sigma_n'< \tau_C$ on $A_n$ for all $n \in \N$, where we used that $\underline{Z}_{\tau_C-} = 0$ on the event $\{\tau_C < \infty\}$. We need to show that $\lim_{n \uparrow \infty} \sigma_n = \tau_C$.  On the event $A_C$, see \eqref{eq:tauC}, we have $\tau_C = \tau_0 = \lim_{n \uparrow \infty} \sigma_n' = \lim_{n \uparrow \infty} \sigma_n$ since $A_C \subset A_n$ for all $n \in \N$. On the event $A_C^c = \bigcup_{n =1}^\infty A_n^c$, we have $\tau_C = \infty = \lim_{n \uparrow \infty} n = \lim_{n \uparrow \infty} \sigma_n$. Hence $(\sigma_n)_{n \in \N}$ is an announcing sequence of $\tau_C$, as claimed.
\end{proof}

If a  semimartingale $Z$  reaches zero continuously, the process $H = \frac{1}{Z_-}\1{Z_-\neq0}$ explodes in finite time, and is therefore not left-continuous. In fact, it is not $Z$--integrable. However, if we view $Z$ as a semimartingale on the stochastic interval $\lc 0,\tau_C\lc$, then $H$ is $Z$--integrable in the sense of stochastic integration on stochastic intervals, as introduced in Section~\ref{S:2}. Thus $H \cdot Z$ exists as a semimartingale on~$\lc0,\tau_C\lc$, which we call the stochastic logarithm of $Z$.

\begin{definition}[Stochastic logarithm]\label{D:stochLog}\footnote{See also the Addendum, which discusses an additional requirement on $Z$ for the definition to be correct.}
Let $\tau$ be a foretellable time and $Z$ be a progressively measurable process such that $\tau \leq \tau_C$ and  such that $Z$ is a semimartingale on $\lc 0,\tau\lc$.
The  semimartingale $\mathcal L(Z)$ on $\lc0,\tau\lc$ defined by
\begin{align*}
\mathcal L( Z ) = \frac{1}{Z_-}\1{Z_- \neq 0} \cdot Z \qquad \text{on $\lc 0,\tau\lc$}
\end{align*}
is called the \emph{stochastic logarithm of $Z$ (on $\lc 0, \tau\lc$)}.\qed
\end{definition}

\subsection{The relationship of stochastic exponentials and logarithms}

\begin{theorem} \label{T:181008}
Let $\tau$ be a foretellable time. We then have the following two statements.
\begin{enumerate}
	\item \footnote{See also the Addendum, which discusses an additional requirement on $Z$ for this statement to be correct.}
	Let $Z$ be a progressively measurable process with $Z_0 =  1$ such that  $\tau \leq \tau_C$,  $Z$ is a semimartingale on $\lc 0,\tau\lc$, and $Z = 0$ on $\lc \tau, \infty\lc$. Then
	\[
		Z = \mathcal E(\mathcal L(Z)).
	\]
	\item  Let $X$ be a semimartingale on $\lc 0, \tau \lc$ with $X_0 = 0$ such that $X$ stays constant after its first jump by $-1$. Then $\mathcal E(X)$ is a semimartingale on $\lc 0, \tau \lc$, does not hit zero continuously strictly before $\tau$, and satisfies
	\[
		X = \mathcal L(\mathcal E(X)) \qquad \text{on $\lc 0,\tau\lc$}.
	\]
	\end{enumerate}
\end{theorem}

\begin{proof} We start by proving (i).
	By assumption and by Definition~\ref{eq:D:stochExp}, both sides are zero on $\lc \tau, \infty\lc$.  Moreover, $Z$ satisfies the equation 
	\[
		Z = 1 + Z_-(Z_-)^{-1} \1{Z_-\neq 0} \cdot Z = 1 + Z_-  \cdot \mathcal{L}(Z) \qquad \text{on $\lc 0, \tau\lc$,}
	\]
	 whose unique solution is $\Ecal(\mathcal L(Z))$ on $\lc0,\tau\lc$.
	 
	 We next prove (ii). Note that $\mathcal E(X)$ is clearly a semimartingale on $\lc 0, \tau\lc$, which also does not hit zero continuously strictly before $\tau$, thanks to Proposition~\ref{P:1'''}. Then the definition of stochastic logarithm along with \eqref{eq:170203.1} yield
	 \begin{align*}
	 	\mathcal L(\mathcal E(X)) &= \frac{1}{\mathcal E(X)_-} \1{\mathcal E(X)_- \neq 0} \cdot \mathcal E(X) = \frac{1}{\mathcal E(X)_-} \1{\mathcal E(X)_- \neq 0}  \mathcal E(X)_- \cdot X \\
		&=  \1{\mathcal E(X)_- \neq 0}  \cdot X = X \qquad \text{on $\lc 0,\tau\lc$},
	 	\end{align*}
		where the last equality follows from the fact that $X$ stays constant after it jumps by $-1$. 
\end{proof}

\subsection{The special case of local supermartingales}
Consider now the case where $X$ is a local supermartingale on $\lc0,\tau\lc$ with $\Delta X \geq -1$. Then $\Ecal(X)$ is also a local supermartingale on $\lc0,\tau\lc$ due to its positivity and \eqref{eq:170203.1}. Moreover, the same argument as in the proof of Proposition~\ref{P:supermg conv} yields  that $\Ecal(X)$ is in fact a supermartingale globally, i.e.~on $\lc0,\infty\lc$.

\begin{remark}\label{R:181009} 
	We can now provide an alternative proof of the inclusion ``$\subset$'' in Lemma~\ref{L:181009} under the additional assumption that $\Delta X > -1$. 
	Thanks to Proposition~\ref{P:1'''}, it suffices to show 
	\[
	\left\{\lim_{t\uparrow\tau} X_t  \text{ does not exist in $\R$} \right\} \subset \left\{\lim_{t\uparrow\tau}\Ecal(X)_t=0\right\}.
	\]
	As in the proof of the inclusion ``$\supset$'', we  deduce
\[
\left\{\lim_{t\uparrow\tau}X_t \text{ does not exist in $\R$}\right\} \subset \bigcap_{m\in\N} \{X \ge -m\}^c \subset \left\{\liminf_{t\uparrow\tau} X_t = -\infty\right\}.
\]
Since \eqref{eq:170206.1} yields $0\le\Ecal(X) \le \bm1_{\lc0,\tau\lc}e^X$ and the limit $\lim_{t\uparrow\tau} \Ecal(X)_t $ exists by Proposition~\ref{P:supermg conv}, the inclusion follows.
\qed
\end{remark}

Nonnegative supermartingales $Z$ can be associated to a probability measure; see Chapter~11 in \cite{Chung:Walsh} in the context of so-called $h$-transforms, or \cite{Perkowski_Ruf_2014} in the general context. Girsanov then provides the drift correct correction for a process $Y$ as the quadratic covariation of $Y$ and $\mathcal L(X)$, namely $[Y, \mathcal L(X)]$. Hence, it is helpful to understand well the connection between a nonnegative supermartingale and its stochastic logarithm.

To this end, we now want to make Theorem~\ref{T:181008} more concrete, namely to work out the relationship of stochastic exponentials and logarithms in the local supermartingale case. The following definition will be helpful.
\begin{definition}[Maximality]\label{D:maximal}
Let $\tau$ be a foretellable time, and let $X$ be a semimartingale on $\lc0,\tau\lc$. We say that $\tau$ is {\em $X$--maximal} if the inclusion $$\{\tau<\infty\}\subset  \left\{\lim_{t\uparrow\tau} X_t  \text{ does not exist in $\R$} \right\}$$ holds almost surely. \qed
\end{definition}

%Thus $\tau$ is $X$--maximal if $X$ does not admit a finite left limit at $\tau$, and therefore cannot be extended continuously up to $\tau$, whenever $\tau$ is finite.

Let now $\mathfrak{Z}$ be the set of all nonnegative supermartingales $Z$ with $Z_0 = 1$. Any such process~$Z$ automatically satisfies $Z = Z^{\tau_0}$. Furthermore, let $\mathfrak L$ denote the set of all stochastic processes~$X$ satisfying the following conditions:
\begin{enumerate}
\item $X$ is a local supermartingale on $\lc0,\tau\lc$ for some foretellable, $X$--maximal time $\tau$.
\item $X_0=0$, $\Delta X\ge -1$ on $\lc0,\tau\lc$, and $X$ is constant after the first time $\Delta X=-1$.
\end{enumerate}

The next theorem extends the classical correspondence between strictly positive local martingales and local martingales with jumps strictly greater than~$-1$. The reader is referred to Proposition~I.5 in \citet{Lepingle_Memin_Sur} and Appendix~A of \citet{K_balance} for related results. In both of these references, the local martingale is not allowed to hit zero continuously.

\begin{theorem}[Relationship of stochastic exponential and logarithm] \label{T:SE}
The stochastic exponential $\mathcal{E}$ is a bijection from $\mathfrak{L}$ to $\mathfrak{Z}$, and its inverse is the stochastic logarithm $\mathcal L$.  Suppose $Z = \mathcal{E}(X)$ for some $Z\in\mathfrak Z$ and $X \in \mathfrak{L}$. Then $\tau = \tau_C$, where $\tau$ is the foretellable $X$--maximal time corresponding to $X$, and $\tau_C$ is given by~\eqref{eq:tauC}.
\end{theorem}

\begin{proof}
By Theorem~\ref{T:181008}, we have  $\Ecal\circ\Lcal={\rm id}$ and $\Lcal\circ\Ecal={\rm id}$.

Next, $\Ecal$ maps each $X\in\mathfrak L$ to some $Z\in\mathfrak Z$ with $\tau_C=\tau$. Since $Z=\Ecal(X)$ is a nonnegative supermartingale with $Z_0=1$, we have $Z\in\mathfrak Z$ and only need to argue that $\tau_C=\tau$. By Theorem~\ref{T:181008}(ii), we have $\tau \leq \tau_C$.  The reverse inequality follows from Proposition~\ref{P:1'''}, Lemma~\ref{L:181009}, and the $X$--maximality of $\tau$.

Further, $\Lcal$ maps each $Z\in\mathfrak Z$ to some $X\in\mathfrak L$ with $\tau=\tau_C$. Indeed, $X = \mathcal{L}(Z)$ is a local supermartingale on $\lc0, \tau_C\lc$ with $X_0 = 0$,  $\Delta X=Z / Z_--1\ge-1$ on $\lc0,\tau_C\lc$, and $X$ is constant after the first time $\Delta X=-1$. It remains to check that $\tau=\tau_C$ is $X$--maximal. To this end, observe that $Z=\Ecal(\Lcal(Z))=\Ecal(X)$. Thus on $\{\tau<\infty\}$ we have $\tau_J=\infty$ and hence $\Delta X>-1$. It follows from Proposition~\ref{P:1'''} and Lemma~\ref{L:181009} that
\begin{align*}
\{\tau < \infty\} &\subset \left\{ \lim_{t\uparrow\tau} Z_t = 0\right\} \cap \{\Delta X>-1\} \subset  \left\{\lim_{t\uparrow\tau}X_t  = -\infty\right\} \cup \left\{[X,X]_{\tau}  = \infty\right\} \\
& \subset  \left\{\lim_{t\uparrow\tau} X_t  \text{ does not exist in $\R$} \right\} ,
\end{align*}
hence $\tau$ is $X$--maximal as claimed.
\end{proof}

\section{Reciprocals of stochastic exponentials}  \label{S:4}

Reciprocals of stochastic exponentials appear naturally in connection with changes of probability measures. We now develop some identities related to such reciprocals. The following function plays an important role:
\begin{equation*}
\phi: (-1,\infty) \to (-1,\infty), \qquad \phi(x) = -1 + \frac{1}{1+x}.
\end{equation*}
Note that $\phi$ is an involution, that is, $\phi(\phi(x))=x$. The following notation is convenient: Given functions $F:\Omega\times\R_+\times\R\to\R$ and $f:\R\to\R$, we write $F\circ f$ for the function $(\omega,t,x)\mapsto F(\omega,t,f(x))$. We now identify the reciprocal of a stochastic exponential or, more precisely, the stochastic logarithm of this reciprocal. Part of the following result is contained in Lemma~3.4 of~\citet{KK}.

\begin{theorem}[Reciprocal of a stochastic exponential] \label{T:reciprocal}
Let $\tau$ be a foretellable time, and let $X$ be a semimartingale on $\lc 0,\tau\lc$. % with $\Delta X \geq -1$. 
Define the semimartingale 
\begin{align}  \label{eq:defN}
	Y = -X + [X, X]^c +  \frac{x^2}{1+x}\oo_{x\neq-1} * \mu^X \qquad \text{on $\lc0,\tau \lc$.}
\end{align}
Then $\mathcal{E}(X) \mathcal{E}(Y) = 1$ on $\lc0,\tau \wedge \tau_J\lc$. Furthermore, for any nonnegative function  $G:\Omega\times\R_+\times\R\to\R_+$ we have 
	\begin{align}  \label{eq:fmu}
		G * \mu^Y = (G\circ \phi) * \mu^X \qquad \text{on $\lc0,\tau \wedge \tau_J\lc$.}
	\end{align}
\end{theorem}

For an alternative, systematic proof of Theorem~\ref{T:reciprocal}, see also \cite{Cerny:Ruf}.

\begin{remark}
	Since $|x^2/(1+x)| \leq 2x^2$ for $|x|\le1/2$, the process $x^2/(1+x) * \mu^X$ appearing in \eqref{eq:defN} is finite-valued on $\lc0,\tau\lc$.\qed
\end{remark}

\begin{remark}\label{R:alternative}
	Since $\phi$ is an involution, the identity \eqref{eq:fmu} is equivalent to
	\[
	F*\mu^X = (F \circ \phi) * \mu^Y \qquad \text{on $\lc0,\tau \wedge \tau_J\lc$}
	\]
	for the nonnegative function $F=G\circ\phi$. \qed
\end{remark}

\begin{proof}[Proof of Theorem~\ref{T:reciprocal}]
Note that, in view of \eqref{eq:defN}, we have
\[
\Delta Y = -\Delta X + \frac{(\Delta X)^2}{1+\Delta X} = \phi(\Delta X) \qquad \qquad \text{on $\lc 0,\tau \wedge \tau_J\lc$}.
\]
This implies \eqref{eq:fmu}.  Now, applying~\eqref{eq:fmu} to the function $G(y) = y-\log|1+y|$ yields
\[
(y-\log|1+y|)*\mu^Y = \left(-1+\frac{1}{1+x} + \log|1+x|\right) * \mu^X \qquad \text{on $\lc 0,\tau \wedge \tau_J\lc$}. 
\]
A direct calculation then gives $\Ecal(Y)=1/\Ecal(X)$ on $\lc0,\tau \wedge \tau_J\lc$. This completes the proof.
\end{proof}

\section{Examples}  \label{S:5}
In this section, we collect some examples to put this note's results into context. We begin with two examples that are rather standard and concern geometric Brownian motion and the stochastic exponential of a one-jump martingale.

\begin{example}
	Let $X \in \mathfrak L$ be Brownian motion.  Then the stopping time $\tau = \infty$ is $X$--maximal, and $\mathcal E(X)_t= e^{X_t-t/2}$ for all $t \geq 0$ is geometric Brownian motion. From \eqref{eq:defN},  $Y_t = t-X_t$ for all $t \geq 0$ satisfies $\mathcal E(X) \mathcal E(Y) = 1$. \qed
\end{example}

\begin{example}
	Let $E$ be a standard exponentially distributed random variable and assume that $\mathbb F$ is the smallest right-continuous filtration such that $\oo_{\lc 0, E\lc}$ is adapted. Let now $X$ be given by $X_t = t \wedge E - \oo_{\{E \leq t\}}$ for all $t \geq 0$. Then $X \in \mathfrak L$, $X$ is a martingale,  $\tau = \infty$ is $X$--maximal, and $\mathcal E(X)_t= e^{t}\oo_{\lc 0, E\lc}$ for all $t \geq 0$. \qed
\end{example}

The next two examples discuss the stochastic exponentials of a random walk and of a time-changed version of it.

\begin{example}   \label{Ex:3}
Let $(\Theta_n)_{n \in \N}$ denote a sequence of independent random variables with $\P(\Theta_n = 1)=\P(\Theta_n = -1) = 1/2$ and let $X_t = \sum_{n=1}^{[t]} \Theta_n$ for all $t \geq 0$ be a standard random walk, where $[t]$ denotes the integer part of $t$. Assume that $\mathbb F$ is the smallest right-continuous filtration such that $X$ is adapted.  Define the stopping time
\begin{align*}
	\rho = \inf \{t \geq 0: \Delta X_t = -1\}.
\end{align*}
Then $\mathcal E(X)_t = \mathcal E(X^\rho)_t = 2^{[t]} \oo_{\lc 0, \rho\lc}$ for all $t \geq 0$ and $X^\rho = \mathcal L (\mathcal E(X)) \in \mathfrak L$.
\qed
\end{example}

\begin{example}
Similarly to Example~\ref{Ex:3}, let $(\Theta_n)_{n \in \N}$ denote a sequence of independent random variables with $\P(\Theta_n = 1/2)=\P(\Theta_n = -1/2) = 1/2$. Let now $X$ be a local martingale on $[0,1)$, given as a random walk that jumps at the deterministic times $(t_n)_{n \in \N}$ where $t_n=1-n^{-1}$. That is, $X_t = \sum_{n\colon t_n \le t}\Theta_n$ for all $t \in [0,1)$. Assume that $\mathbb F$ is the smallest right-continuous filtration such that $X$ is adapted. Then the deterministic stopping time $\tau = 1$ is $X$--maximal; hence $X \in \mathfrak L$. Moreover, the stochastic exponential $Z = \mathcal{E}(M) \in \mathfrak Z$ is given by
\[
Z_t = \exp\left( \sum_{n: t_n\le t} \log(1+\Theta_n) \right) = \prod_{n: t_n\le t} (1+\Theta_n).
\]

Theorem~\ref{T:SE} yields that $\tau_C = 1$, where $\tau_C$ was defined in \eqref{eq:tauC}. Alternatively, the strong law of large numbers and the fact that $\E[\log(1+\Theta_1)]<0$ imply $\P(\lim_{t\uparrow 1}Z_t = 0)=1$. 
This illustrates that $\tau_C$ in \eqref{eq:tauC} can be finite, even if the local martingale $Z$ has no continuous component.

Define next the semimartingale Y on $\lc 0, 1\lc$ by 
\begin{align*}
	Y_t = -X_t + \frac{x^2}{1+x} * \mu^X_t =  \left(-x+ \frac{x^2}{1+x}\right) * \mu^X_t = \frac{-x}{1+x} * \mu^X_t  \qquad \text{for all $t \in [0,1)$}.
\end{align*}
Theorem~\ref{T:reciprocal} yield that $\mathcal E(X) \mathcal E(Y) = 1$ on $\lc 0, 1\lc$. Indeed, using \eqref{eq:170206.1}, we get
\begin{align*}
	\mathcal E(Y) &=  \oo_{\lc0,1\lc} \exp\left( \log(1+y) * \mu^Y\right) =   \oo_{\lc0,1\lc} \exp\left( \log\left(1+\frac{-x}{1+x}\right) * \mu^X\right)\\
		&= \oo_{\lc0,1\lc} \exp\left(- \log\left(1+x\right) * \mu^X\right),
\end{align*}
which confirms this claim.
\qed
\end{example}

The last example interprets Brownian motion starting in one and stopped when hitting zero as a stochastic exponential and discusses the corresponding stochastic logarithm.

\begin{example}
	Let $B$ be Brownian motion starting in zero, define the stopping time
\begin{align*}
	\rho = \inf \{t \geq 0: B_t = -1\}
\end{align*}	
	and the nonnegative martingale $Z = 1+B^\rho \in \mathfrak Z$. That is, $Z$ is Brownian motion started in one and stopped as soon as it hits zero.  We now compute
	\begin{align*}
		X = \mathcal L(Z)  = \frac{1}{Z}\1{Z>0} \cdot Z =  \frac{1}{1+B} \cdot B  \qquad \text{on $\lc 0,\rho\lc$}.
	\end{align*}
	Note that $\rho$ is indeed $X$--maximal by Theorem~\ref{T:SE}. \qed
\end{example}

\bibliography{aa_bib}{}
\bibliographystyle{apalike}

\section*{Addendum}
After this article was published we realized that we omitted an important assumption in Definition~\ref{D:stochLog}. Indeed, for the definition to make sense we must assume that $Z$ is absorbed in zero after jumping to zero; i.e., $Z = Z^{\tau_J}$ in the notation of \eqref{eq:tauC}. Without this assumption, the stochastic integral in the definition of the stochastic logarithm $\mathcal L(Z)$ might not exist. As an example, consider the deterministic process $Z = (Z_t)_{t \geq 0}$ given by $Z_t = \1{t < 1} + (1-t) \1{t \geq 1}$, along with $\tau = \infty$. In this case we  have $Z_- \neq 0$, but $1/Z_-$ is not integrable with respect to $Z$.  

For this reason, the first sentence of Definition~\ref{D:stochLog} should read as follows:
\begin{quote}
Let $\tau$ be a foretellable time and $Z$ be a progressively measurable process such that $\tau \leq \tau_C$ and  such that $Z$ is a semimartingale on $\lc 0,\tau\lc$ \emph{and such that $Z = Z^{\tau_J}$. }
\end{quote}

Similarly, the first sentence in Theorem~\ref{T:181008}(i) should read as follows:
\begin{quote}
	Let $Z$ be a progressively measurable process with $Z_0 =  1$ such that  $\tau \leq \tau_C$,  $Z$ is a semimartingale on $\lc 0,\tau\lc$, \emph{$Z = 0$ on $\lc \tau, \infty\lc$, and $Z^{\tau_J}= Z$}.
\end{quote}

No further changes are required.

\end{document}